\documentclass[12pt]{amsart}

\usepackage{CJKutf8}
\usepackage{amssymb}
\usepackage{eucal}
\usepackage{color}
\usepackage{arydshln}
\usepackage{verbatim}

\usepackage{enumitem}
\usepackage{multirow,bigdelim}
\usepackage{amsfonts}
\usepackage{dsfont}
\usepackage{amsmath}
\usepackage{mathrsfs}
\usepackage{todonotes}
\usepackage{hyperref}
\usepackage{multimedia}
\usepackage{graphicx}
\usepackage{indentfirst}
\usepackage{bm}
\usepackage{amsthm}
\usepackage{mathrsfs}
\usepackage{latexsym}
\usepackage{amsmath}
\usepackage{amssymb}
\usepackage{microtype}
\usepackage{relsize}
\usepackage{hyperref}

\DeclareSymbolFont{SY}{U}{psy}{m}{n}
\DeclareMathSymbol{\emptyset}{\mathord}{SY}{'306}

\theoremstyle{plain}

\newtheorem{thm}{Theorem}[section]
\newtheorem{cor}[thm]{Corollary}
\newtheorem{lem}[thm]{Lemma}
\newtheorem{prop}[thm]{Proposition}
\theoremstyle{definition}
\newtheorem{defn}[thm]{Definition}

\numberwithin{equation}{section}

\begin{document}
\title[Completely Bounded Representations And Connes Embedding Problem]{Completely Bounded Representations Into Von Neumann Algebras And Connes Embedding Problem}
\author{Junsheng Fang}
\address{School of Mathematical Sciences, Hebei Normal University, Shijiazhuang, Hebei, 050024, China}
\email{junshengfang@hotmail.com}

\author{Chunlan Jiang}
\address{School of Mathematical Sciences, Hebei Normal University, Shijiazhuang, Hebei, 050024, China}
\email{cljiang@hebtu.edu.cn}

\author{Liguang Wang}
\address{School of Mathematical Sciences, Qufu Normal University, Qufu, Shandong, 273165, China}
\email{wangliguang0510@163.com}

\author{Yanli Wang}
\address{School of Mathematical Sciences, Hebei Normal University, Shijiazhuang, Hebei, 050024, China}
\email{ylwangmathematics@outlook.com}
\maketitle
\begin{abstract}
In this paper, we prove that if $\mathcal{A}$ is a unital separable $C^*$-algebra, $\mathcal{M}$ is a von Neumann algebra which has the Kirchberg's quotient weak expectation property (QWEP), and $\phi:\, \mathcal{A}\rightarrow \mathcal{M}$ is a unital completely bounded representation, then there is an invertible operator $S\in \mathcal{M}$ such that $S\phi(\cdot) S^{-1}$ is a $\ast$-representation. On the other hand, Gilles Pisier proved the following result: a unital $C^*$-algebra $\mathcal{A}$ is nuclear if and only if for every unital completely bounded representation $\phi$ of $\mathcal{A}$ into an arbitrary von Neumann algebra $\mathcal{M}$ there is an invertible operator $S\in \mathcal{M}$ such that $S\phi(\cdot) S^{-1}$ is a $\ast$-representation. This implies that there exist von Neumann algebras which are not QWEP. Eberhard Kirchberg showed that every von Neumann algebra has QWEP if and only if every tracial von Neumann algebra embeds into the ultrapower $\mathcal{R}^w$ of the hyperfinite type ${\rm II}_1$ factor $\mathcal{R}$. This provides a negative answer to the Connes Embedding Problem. This paper  relies on previous work of Gilles Pisier and Florin Pop.

\end{abstract}

\section{Introduction}

In a paper from 1955, Richard V.Kadison~\cite{Kad} considered the following problem: Is any bounded, non-selfadjoint representation $\phi$ of a $C^*$-algebra $\mathcal{A}$ on a Hilbert space $\mathcal{H}$ similar to a $\ast$-representation? i.e., does there exist an invertible operator $S\in \mathcal{B}(\mathcal{H})$, such that $S\phi(\cdot)S^{-1}$ is a $\ast$-representation of $\mathcal{A}$?  The similarity problem is still open, although a number of partial results have been obtained.  Using the deep results of Alain Connes~\cite{Connes} on injective von Neumann algebras and the characterization of nuclear $C^*$-algebras given in~\cite{C-E}, one can prove that every bounded representation of a nuclear $C^*$-algebra is similar to a $\ast$-representation~\cite{Bunce,Christensen}. The similarity problem for bounded cyclic representations of arbitrary $C^*$-algebras was firstly considered by B. A. Barnes~\cite{Barnes}. He proved that when $\phi$ is a bounded representation of a $C^*$-algebra $\mathcal{A}$ on a Hilbert space $\mathcal{H}$, such that $\phi(\mathcal{A})$ and $\phi(\mathcal{A})^*$ have a common cyclic vector $\xi$, then there exists a closed, injective operator $S$ on $\mathcal{H}$, such that $S$ and $S^{-1}$ are densely defined, and such that $x\rightarrow \overline{S\phi(x)S^{-1}}$ is a $\ast$-representation of $\mathcal{A}$. J. W.  Bunce~\cite{Bunce} sharpened this by proving that $S$ can be chosen bounded, but not necessarily with bounded inverse. He also removed the condition that $\xi$ is cyclic for $\phi(\mathcal{A})^*$. Erik Christensen~\cite{Christensen} proved  that every irreducible bounded representation of a $C^*$-algebra is similar to a $\ast$-representation. Finally, Uffe Haagerup~\cite{Haagerup} proved that every bounded representation of a $C^*$-algebra with a cyclic vector $\xi$ is similar to a $\ast$-representation. Furthermore, the following  theorem is also proved.

\begin{thm}${\rm (Haagerup-Hadwin-Wittstock)}$\cite{Haagerup, Hadwin, Wit}\label{Haagerup}
Let \(\phi\) be a bounded non-degenerate representation of a \(C^{*}\)-algebra \(\mathcal{A}\) on a Hilbert space \(\mathcal{H}\). Then the following are equivalent:

\begin{enumerate}
    \item [{\rm (1)}] \(\phi\) is a completely bounded representation.
    \item [{\rm (2)}] \(\phi\) is similar to a \(*\)-representation.
\end{enumerate}
\end{thm}
Later on, Erik Christensen~\cite{Chr} proved every bounded representation of a type ${\rm II}_1$ factor with property $\Gamma$ is similar to a $\ast$-representation. Very recently, Florin Pop~\cite{Pop} proved that every bounded representation of $\mathcal{M}\bar{\otimes}\mathcal{N}$ of type ${\rm II}_1$ factors is similar to a $\ast$-representation provided $\mathcal{M}$ or $\mathcal{N}$ can be embedded into the ultrapower algebra $\mathcal{R}^w$ of the hyperfinite type ${\rm II}_1$ factor $\mathcal{R}$. The following important theorem  (\cite{Pop}, Proposition 3.1) plays a key role in Florin Pop's work, which also plays a key role in our (both) proofs of the main result (Theorem~\ref{Main}).

\begin{thm}{\rm (Pop)}\cite{Pop}\label{Pop}
Let \(\mathcal{A}_{1},\mathcal{A}_{2},\mathcal{B}_{1},\mathcal{B}_{2}\) be \(C^{*}\)-algebras such that \(\mathcal{A}_{1}\) has the LP and \(\mathcal{B}_{1}\) is QWEP. If \(\alpha:\mathcal{A}_{1}\to \mathcal{B}_{1}\) is a completely bounded homomorphism and \(\varphi:\mathcal{A}_{2}\to \mathcal{B}_{2}\) is a unital completely positive map {\rm (u.c.p.)} map, then the map \(\alpha\otimes\varphi\) extends to a completely bounded map from \(\mathcal{A}_{1}\otimes_{max}\mathcal{A}_{2}\) with values in \(\mathcal{B}_{1}\otimes_{max}\mathcal{B}_{2}\) satisfying \(\|\alpha\otimes\varphi\|_{cb}\leq\|\alpha\|_{cb}\).
\end{thm}

In this paper, we will consider the following generalized version of Kadison's Similarity Problem: If $\phi$ is a  bounded representation from a $C^*$-algebra $\mathcal{A}$ into a von Neumann algebra $\mathcal{M}$, does there exist an invertible element $S\in \mathcal{M}$ such that $S\phi(\cdot)S^{-1}$ is a $\ast$-homomorphism? This problem was firstly studied in~\cite{V-Z} and the authors proved that if $\phi$ is a bounded representation from a $C^*$-algebra into a finite von Neumann algebra, then $\phi$ is similar to a $\ast$-represenation. Along the line, Gilles Pisier proved the following remarkable result~(\cite{Pis} Theorem 1, ~\cite{Oz} Theorem 2.11).

\begin{thm}{\rm (Pisier)}\cite{Pis}\label{Pis}
A unital \(C^{*}\)-algebra \(\mathcal{A}\) is nuclear if and only if for every unital completely bounded homomorphism $\phi$ from $\mathcal{A}$ into an arbitrary von Neumann algebra $\mathcal{M}$ is similar to a $\ast$-representation, i.e.,  there exists an invertible operator \(S \in \mathcal{M}\) such that \(S\phi(\cdot)S^{-1}\) is a \(*\)-representation.
\end{thm}

 The following result is our main theorem.

\begin{thm}\label{Main}
Let \(\mathcal{A}\) be a unital separable \(C^{*}\)-algebra, and \(\mathcal{M}\) a von Neumann algebra with QWEP. Let \(\phi: \mathcal{A} \to \mathcal{M}\) be a unital completely bounded representation. Then there exists a positive invertible operator \(S \in \mathcal{M}\) such that \(S\phi(\cdot)S^{-1}\) is a \(*\)-representation.
\end{thm}

Combining Theorem~\ref{Pis} and Theorem~\ref{Main}, the existence of non nuclear unital separable $C^*$-algebras implies that there exist von Neumann algebras which are not QWEP. This provides a negative answer to Connes Embedding Problem due to Kirchberg's work. The Connes Embedding Problem~\cite{Connes}  states: every tracial von Neumann algebra embeds (in a trace-preserving way) in an ultrapower $\mathcal{R}^w$  of the hyperfinite ${\rm II}_1$ factor $\mathcal{R}$.  The Connes Embedding Problem is considered as one of the
most important open problems in the field of operator algebras.  It
turns out, most notably by Kirchberg's seminal work~\cite{Kir}, that the Connes Embedding Problem is equivalent to a variety of other important conjectures (see Proposition 13.1 of~\cite{pisier}), which touches
most of the subfields of operator algebras, and also some other branches of mathematics
such as noncommutative real algebraic geometry  and quantum information theory. Among these conjectures is the QWEP conjecture: does every von Neumann algebra have QWEP property (see Theorem 14.1 of~\cite{pisier}). In~\cite{Ji}, the authors proved a landmark
theorem in quantum complexity theory known as ${\rm MIP}^*={\rm RE}$. As an application, they proved that the answer to the Connes Embedding Problem is false. In this paper, we give an ``operator algebraic" proof of this result.

This paper is organized as follows. Section 2 contains some preliminary results. In section 3, we prove the main theorem. In section 4, we give another approach to prove the main theorem under the assumption that $\mathcal{M}$ is separable (with separable predual). Section 5 contains some ideas might be useful to attack the original Kadison's similarity problem. For the theory of operator spaces, $C^*$-algebras and von Neumann algebras, we refer to the books~\cite{pisier,B-O,Kadison}.

\section{Preliminary results}

For a proof of the following theorem, we refer to~\cite{Kadison}.
\begin{thm}\label{Kadison}
Let \(\mathcal{A}\) and \(\mathcal{B}\) be \(C^{*}\)-algebras. Then there exists a \(C^{*}\)-norm $\displaystyle{\otimes}_{\max}$ on \(\mathcal{A} \odot \mathcal{B}\) such that for every \(s \in \mathcal{A} \odot \mathcal{B}\) and for every \(C^{*}\)-seminorm \(\alpha\) on \(\mathcal{A} \odot \mathcal{B}\), we have \(\|s\|_{\max} \ge \alpha(s)\). Moreover, if \(s = \sum_{j=1}^n a_j \otimes b_j\), then
\[
\|s\|_{\max} = \sup \Bigl\| \sum_{j=1}^n \varphi(a_j) \psi(b_j) \Bigr\|,
\]
where the supremum is taken over all commuting pairs \((\varphi, \psi)\) of \(*\)-representations of \(\mathcal{A}\) and \(\mathcal{B}\).
\end{thm}

The following definition was introduced by C. Lance~\cite{Lan}.
\begin{defn}
Let \(\mathcal{A}\) be a \(C^{*}\)-algebra and \(\pi:\mathcal{A} \to \mathcal{B}(\mathcal{H}\)) a faithful unital \(*\)-representation. We say that \(\mathcal{A}\) has the \emph{weak expectation property} (WEP) if there exists a unital completely positive (u.c.p.) map \(\varphi: \mathcal{B}(\mathcal{H}) \to \pi(\mathcal{A})''\) such that \(\varphi(\pi(a)) = \pi(a)\) for every \(a \in \mathcal{A}\).
\end{defn}

\begin{defn}
Let \(\mathcal{A}\) be a \(C^{*}\)-algebra. We say that \(\mathcal{A}\) has the \emph{quotient weak expectation property} (QWEP) if \(\mathcal{A}\) is a quotient of a \(C^{*}\)-algebra with the WEP.
\end{defn}

We collect some properties of QWEP algebras~\cite{Ozawa}.
\begin{thm}\label{Ozawa1}
We have the following results.
\begin{enumerate}
    \item[{\rm (1)}] If \(\mathcal{A}_{i}\) is QWEP for all \(i\in I\), then so is \(\prod_{i\in I}\mathcal{A}_{i}\).
    \item[{\rm (2)}] If \(\mathcal{A}\subset \mathcal{B}\) is weakly cp complemented and \(\mathcal{B}\) is QWEP, then so is \(\mathcal{A}\).
    \item[{\rm (3)}] Let \((\mathcal{A}_{i})_{i\in I}\) be an increasing net of {\rm (possibly non-unital)} \(C^{*}\)-subalgebras in \(\mathcal{A}\) {\rm (resp. \(\mathcal{M}\))} whose union is dense in norm {\rm (resp. weak*)} topology. If all \(\mathcal{A}_{i}\) are QWEP, then so is \(\mathcal{A}\) {\rm (resp. \(\mathcal{M}\))}.
    \item[{\rm (4)}] A \(C^{*}\)-algebra \(\mathcal{A}\) is QWEP if and only if the second dual \(\mathcal{A}^{**}\) is QWEP.
    \item[{\rm (5)}] If \(\mathcal{A}\) is QWEP and \(\mathcal{B}\) is nuclear, then \(\mathcal{A}\otimes_{\min}\mathcal{B}\) is QWEP. If \(\mathcal{M}\) and \(\mathcal{N}\) are QWEP, then so is \(\mathcal{M}\bar{\otimes}\mathcal{N}\).
    \item[{\rm (6)}] If \(\mathcal{A}\) {\rm (resp. \(\mathcal{M}\))} is QWEP and \(\alpha\) is an action of an amenable group \(\Gamma\), then \(\Gamma\ltimes_{\alpha}\mathcal{A}\) {\rm (resp. \(\Gamma\ltimes_{\alpha}\mathcal{M}\))} is QWEP.
    \item[{\rm (7)}] The commutant \(\mathcal{M}^{\prime}\) is QWEP if and only if \(\mathcal{M}\) is QWEP.
    \item[{\rm (8)}] Let \(\mathcal{M}=\int^{\oplus}\mathcal{M}(\gamma)\,d\gamma\) be the direct integral of separable von Neumann algebras. Then, \(\mathcal{M}\) is QWEP if and only if \(\mathcal{M}(\gamma)\) are QWEP for almost every \(\gamma\).
\end{enumerate}
\end{thm}

\begin{defn}
Let \(\mathcal{A}\) and \(\mathcal{B}\) be \(C^{*}\)-algebras, and let \(\mathcal{J}\) be a closed two-sided ideal of \(\mathcal{B}\). We say that \(\mathcal{A}\) has the \emph{lifting property} (LP) if for every u.c.p. map \(\varphi: \mathcal{A} \to \mathcal{B}/\mathcal{J}\), there exists a u.c.p. map \(\psi: \mathcal{A} \to \mathcal{B}\) such that \(\varphi = \pi \circ \psi\), where \(\pi: \mathcal{B} \to \mathcal{B}/\mathcal{J}\) is the quotient map.
\end{defn}

For a proof of the following result, we refer to~\cite{Ozawa}.
\begin{thm}\label{Ozawa}
The full \(C^{*}\)-algebra \(C^{*}(F_\infty)\) of a countable free group \(F_\infty\) has the LP.
\end{thm}

\section{Main theorem}

\begin{lem}\label{three equivalences}
Let \(\mathcal{A}\) be a \(C^{*}\)-algebra, \(\mathcal{M}\) a von Neumann algebra, and \(\phi: \mathcal{A} \to \mathcal{M}\) a completely bounded representation. Then the following conditions are equivalent:

\begin{enumerate}
    \item [{\rm (1)}]There exists an invertible operator \(h \in \mathcal{M}\) such that \(h\phi(\cdot)h^{-1}\) is a \(*\)-representation.
    \item [{\rm (2)}]There exists a positive invertible operator \(S \in \mathcal{M}\) such that \(S\phi(\cdot)S^{-1}\) is a \(*\)-representation.
    \item [{\rm (3)}]There exists a positive invertible operator \(T \in \mathcal{M}\) such that for every \(x \in \mathcal{A}\), \(T\phi(x) = \phi(x^{*})^{*}T\).
\end{enumerate}
\end{lem}

\begin{proof}
(3) \(\Rightarrow\) (2). Suppose there exists a positive invertible operator \(T \in \mathcal{M}\) such that for every \(x \in \mathcal{A}\), \(T\phi(x) = \phi(x^{*})^{*}T\). Let \(S = T^{1/2}\). Then \(S \in \mathcal{M}\) and
\[
S\phi(x)S^{-1} = S^{-1}\phi(x^{*})^{*}S.
\]
Set \(\psi(x) = S\phi(x)S^{-1}\). Then
\[
\psi(x^{*}) = S\phi(x^{*})S^{-1} = S^{-1}\phi(x)^{*}S = (S\phi(x)S^{-1})^{*} = \psi(x)^{*},
\]
so \(S\phi(\cdot)S^{-1}\) is a \(*\)-representation.

(2) \(\Rightarrow\) (1) is obvious.

(1) \(\Rightarrow\) (3). Suppose there exists an invertible operator \(h \in \mathcal{M}\) such that \(h\phi(\cdot)h^{-1}\) is a \(*\)-representation. Set \(\psi(x) = h\phi(x)h^{-1}\). Then
\[
h\phi(x^{*})h^{-1} = \psi(x^{*}) = \psi(x)^{*} = (h^{*})^{-1}\phi(x)^{*}h^{*}.
\]
which implies
\[
h^*h\phi(x^*) = \phi(x)^* h^* h.
\]
Replacing \(x\) by \(x^*\) yields
\[
h^*h\phi(x) = \phi(x^*)^* h^* h.
\]
Let \(T = h^*h\). Then the positive operator \(T \in \mathcal{M}\). Moreover, since \(T = h^*h\) is invertible, there exists a positive invertible operator \(T\in\mathcal{M}\) such that \(T\phi(x) = \phi(x^*)^* T\).
\end{proof}

\begin{defn}
A von Neumann algebra \(\mathcal{M} \subseteq \mathcal{B}(\mathcal{H})\) is called \emph{injective} if there exists a conditional expectation \(E: \mathcal{B}(\mathcal{H}) \to \mathcal{M}\).
\end{defn}

\begin{thm}
Let \(\mathcal{A}\) be a \(C^*\)-algebra, \(\mathcal{M}\) an injective von Neumann algebra, and \(\phi: \mathcal{A} \to \mathcal{M}\) a completely bounded representation. Then there exists a positive invertible operator \(S \in \mathcal{M}\) such that \(S\phi(\cdot)S^{-1}\) is a \(*\)-representation.
\end{thm}
\begin{proof}
Without loss of generality, assume that \(\mathcal{M}\) acts faithfully on a Hilbert space \(\mathcal{H}\), so that \(\phi\) is a completely bounded representation of the \(C^*\)-algebra \(\mathcal{A}\) on \(\mathcal{H}\). By Theorem \ref{Haagerup}, \(\phi\) is similar to a \(*\)-representation in $\mathcal{B}(\mathcal{H})$.  By Lemma \ref{three equivalences}, there exists a positive invertible operator \(h \in \mathcal{B}(\mathcal{H})\) such that
\[
h\phi(x) = \phi(x^*)^* h.
\]
Since \(\mathcal{M}\) is an injective von Neumann algebra, there exists a conditional expectation \(E: \mathcal{B}(\mathcal{H}) \to \mathcal{M}\). Let \(T = E(h) \in \mathcal{M}\). Since \(h\) is a positive invertible bounded operator, there exist constants \(0 < \alpha < \beta\) such that \(\alpha I \leq h \leq \beta I\). Hence,
\[
\alpha I \leq T = E(h) \leq \beta I.
\]
So \(T\) is also an invertible positive bounded operator. Moreover,
\[
T\phi(x) = E(h\phi(x)) = E(\phi(x^*)^* h) = \phi(x^*)^* T, \quad \forall x \in \mathcal{A}.
\]
By Lemma \ref{three equivalences}, there exists a positive invertible operator \(S = T^{1/2} \in \mathcal{M}\) such that \(S\phi(\cdot)S^{-1}\) is a \(*\)-representation.
\end{proof}

The following result is well known. We include a proof for completeness.
\begin{thm}
Let \(\mathcal{A}\) be a nuclear \(C^{*}\)-algebra, \(\mathcal{M}\) a von Neumann algebra, and \(\phi: \mathcal{A} \to \mathcal{M}\) a bounded representation. Then there exists a positive invertible operator \(S \in \mathcal{M}\) such that \(S\phi(\cdot)S^{-1}\) is a \(*\)-representation.
\end{thm}
\begin{proof}
Without loss of generality, assume that \(\mathcal{M}\) acts faithfully on a Hilbert space \(\mathcal{H}\), so that \(\phi\) is a bounded representation of the nuclear \(C^{*}\)-algebra \(\mathcal{A}\) on \(\mathcal{H}\). By~\cite{Christensen}, $\phi$ is similar to a $\ast$-representation in $\mathcal{B}(\mathcal{H})$, i.e., there exists a positive invertible operator \(T^{-1} \in \mathcal{B}(\mathcal{H})\) such that \(T^{-1}\phi(\cdot)T\) is a \(*\)-representation \(\pi(\cdot)\). Hence,
\[
\phi(a) = T\pi(a)T^{-1}, \quad \forall a \in \mathcal{A}.
\]
Define a representation of the $C^*$-algebra $\mathcal{A} \otimes_{\min} \mathcal{M}'$ on the Hilbert space $\mathcal{H}$ by
\[
\Phi: \mathcal{A} \otimes_{\min} \mathcal{M}' \longrightarrow \mathcal{B}(\mathcal{H}), \]
\[\Phi\Bigl(\sum_{i=1}^n a_i \otimes x_i\Bigr) := \sum_{i=1}^n \phi(a_i)x_i \quad \forall a_i \in \mathcal{A}, x_i \in \mathcal{M}'.
\]
$\Phi$ is well-defined. Note that $\phi(a_i)x_i = x_i\phi(a_i)$ and $\phi(a_i) = T\pi(a_i)T^{-1}$. Hence, $\pi(a_i)T^{-1}x_iT = T^{-1}x_iT\pi(a_i)$. By the definition of nuclear $C^*$-algebras and Theorem \ref{Kadison}, we have
\begin{align*}
\Bigl\|\Phi\Bigl(\sum_{i=1}^n a_i \otimes x_i\Bigr)\Bigr\|
&= \Bigl\|\sum_{i=1}^n \phi(a_i)x_i\Bigr\| \\
&= \Bigl\|\sum_{i=1}^n T\pi(a_i)T^{-1}x_iTT^{-1}\Bigr\| \\
&\leq \|T\| \|T^{-1}\|\Bigl\|\sum_{i=1}^n \pi(a_i)T^{-1}x_iT\Bigr\| \\
&\leq \|T\| \|T^{-1}\|\Bigl\|\sum_{i=1}^n a_i \otimes T^{-1}x_iT\Bigr\|_{\mathcal{A} \otimes_{\max} \pi(\mathcal{A})'} \\
&= \|T\| \|T^{-1}\|\Bigl\|\sum_{i=1}^n a_i \otimes T^{-1}x_iT\Bigr\|_{\mathcal{A} \otimes_{\min} \pi(\mathcal{A})'} \\
&= \|T\| \|T^{-1}\|\Bigl\|\sum_{i=1}^n a_i \otimes T^{-1}x_iT\Bigr\|_{\mathcal{A} \otimes_{\min} \mathcal{B}(\mathcal{H})} \\
&\leq \|T\|^2 \|T^{-1}\|^2\Bigl\|\sum_{i=1}^n a_i \otimes x_i\Bigr\|_{\mathcal{A} \otimes_{\min} \mathcal{B}(\mathcal{H})} \\
&= \|T\|^2 \|T^{-1}\|^2\Bigl\|\sum_{i=1}^n a_i \otimes x_i\Bigr\|_{\mathcal{A} \otimes_{\min} \mathcal{M}'}.
\end{align*}
Thus, $\Phi$ is bounded on a dense subset of $\mathcal{A} \otimes_{\min} \mathcal{M}'$, and hence $\Phi$ is bounded on $\mathcal{A} \otimes_{\min} \mathcal{M}'$.

Consider the representation
\[
\Phi^{(n)}: \mathcal{A} \otimes_{\min} \mathcal{M}' \otimes_{\min} M_n(\mathbb{C}) \longrightarrow \mathcal{B}(\mathcal{H}) \otimes_{\min} M_n(\mathbb{C})\]
defined by
\[\Phi^{(n)}\Bigl(\sum_{i=1}^k (a^i_{rs}) \otimes x_i\Bigr) := \sum_{i=1}^k (\phi(a^i_{rs})x_i) \quad \forall a^i_{rs} \in \mathcal{A}, x_i \in \mathcal{M}'.\]
Similarly,
\begin{align*}
\Bigl\|\Phi^{(n)}\Bigl(\sum_{i=1}^k (a^i_{rs}) \otimes x_i\Bigr)\Bigr\|
&= \Bigl\|\sum_{i=1}^k (\phi(a^i_{rs})x_i)\Bigr\| \\
&= \Bigl\|\sum_{i=1}^k (T\pi(a^i_{rs})T^{-1}x_iTT^{-1})\Bigr\| \\
&\leq  \|T\|\|T^{-1}\Bigl\|\sum_{i=1}^k (\pi(a^i_{rs}))(T^{-1}\otimes I_n)(x_i\otimes I_n)(T\otimes I_n)\Bigr\| \\
&\leq \|T\|\|T^{-1}\|\Bigl\|\sum_{i=1}^k (a^i_{rs}) \otimes T^{-1}x_iT\Bigr\|_{M_n(\mathcal{A}) \otimes_{\max} \pi(\mathcal{A})'} \\
&= \|T\|\|T^{-1}\|\Bigl\|\sum_{i=1}^k (a^i_{rs}) \otimes T^{-1}x_iT\Bigr\|_{M_n(\mathcal{A}) \otimes_{\min} \pi(\mathcal{A})'} \\
&= \|T\|\|T^{-1}\|\Bigl\|\sum_{i=1}^k (a^i_{rs}) \otimes T^{-1}x_iT\Bigr\|_{M_n(\mathcal{A}) \otimes_{\min} \mathcal{B}(\mathcal{H})} \\
&\leq \|T\|^2\|T^{-1}\|^2\Bigl\|\sum_{i=1}^k (a^i_{rs}) \otimes x_i\Bigr\|_{M_n(\mathcal{A}) \otimes_{\min} \mathcal{B}(\mathcal{H})} \\
&= \|T\|^2\|T^{-1}\|^2\Bigl\|\sum_{i=1}^k (a^i_{rs}) \otimes x_i\Bigr\|_{M_n(\mathcal{A}) \otimes_{\min} \mathcal{M}'}.
\end{align*}
Thus, $\Phi$ is completely bounded on a dense subset of $\mathcal{A} \otimes_{\min} \mathcal{M}'$, and hence $\Phi$ is a completely bounded representation of $\mathcal{A} \otimes_{\min} \mathcal{M}'$.

By Theorem \ref{Haagerup}, $\Phi$ is similar to a $*$-representation in $\mathcal{B}(\mathcal{H})$.  By Lemma \ref{three equivalences}, there exists a positive invertible operator $h \in \mathcal{B}(\mathcal{H})$ such that
\[
h\Phi\Bigl(\sum_{i=1}^n a_i \otimes x_i\Bigr) = \Phi\Bigl(\sum_{i=1}^n a_i^* \otimes x_i^*\Bigr)^* h, \quad \forall a_i \in \mathcal{A}, \; x_i \in \mathcal{M}'.
\]
In particular, for $x \in \mathcal{M}'$, we have
\[
h\Phi(1 \otimes x) = \Phi(1 \otimes x^*)^* h.
\]
Since $\Phi(1 \otimes x) = x$ and $\Phi(1 \otimes x^*) = x^*$, this implies $hx = xh$ for all $x \in \mathcal{M}'$. Therefore, $h \in \mathcal{M}''$. By the Double Commutant Theorem, $\mathcal{M}'' = \mathcal{M}$, hence $h \in \mathcal{M}$. By Lemma \ref{three equivalences}, there exists a positive invertible operator $S = h^{1/2} \in \mathcal{M}$ such that for all $a \in \mathcal{A}$,
$$
h\phi(a) = \phi(a^*)^* h,
$$
which implies that $S\phi(\cdot)S^{-1}$ is a $*$-representation.
\end{proof}

\begin{thm}
Let $\mathcal{A}$ be a $C^{*}$-algebra, $\mathcal{M}$ a von Neumann algebra, and $\varphi: \mathcal{A}\rightarrow\mathcal{M}$ a  bounded representation. Then the following are equivalent:
\begin{enumerate}
\item [{\rm (1)}]$\varphi$ is similar to a $*$-representation in $\mathcal{M}$;
\item [{\rm (2)}]$\Phi$ is a completely bounded representation, where $\Phi: \mathcal{A}\otimes_{max}\mathcal{M}'\rightarrow\mathcal{B}(\mathcal{H})$ is defined by
$$\Phi(\sum\limits_{i=1}^na_{i}\otimes x_{i}):=\sum\limits_{i=1}^n\varphi(a_{i})x_{i}\quad\forall~ a_{i}\in\mathcal{A},x_{i}\in\mathcal{M'};$$
\item [{\rm (3)}]$\varphi=\varphi_{1}-\varphi_{2}+i(\varphi_{3}-\varphi_{4})$, where $\varphi_{1}, \varphi_{2}, \varphi_{3}, \varphi_{4}\in CP(\mathcal{A},\mathcal{M})$, the set of completely positive maps from $\mathcal{A}$ to $\mathcal{M}$;
\item [{\rm (4)}]There exist $\phi_{1}, \phi_{2}\in CP(\mathcal{A},\mathcal{M})$ such that
$\Psi: \mathcal{A}\rightarrow\mathcal{M}\otimes\mathcal{M}_{2}(\mathbb{C})$ defined by
\begin{equation*}
\Psi(a)=\left ( \begin{matrix}\phi_{1}(a) &\varphi(a^{*})^{*}\\
\varphi(a)&\phi_{2}(a)
\end{matrix}\right )\quad\forall~a\in\mathcal{A}
\end{equation*}
is a completely positive map;
\item [{\rm (5)}]There exists a completely positive map $\phi: \mathcal{A}\rightarrow\mathcal{M}$ and $T, S\in\mathcal{M}$ such that  $\varphi(a)=T^{*}\phi(a)S$, $\forall a\in\mathcal{A}$;
\item [{\rm (6)}]There exist finitely many completely positive maps $\phi_{i}: \mathcal{A}\rightarrow\mathcal{M}$ and $T_{i}, S_{i}\in\mathcal{M}$ such that $\varphi(a)=\sum\limits_{i=1}^{n}T_{i}^{*}\phi_{i}(a)S_{i}$, $\forall a\in \mathcal{A}$;
\item [{\rm (7)}]There exist finitely many $*$-representations $\pi_{i}: \mathcal{A}\rightarrow\mathcal{M}$, positive invertible operators $T_{i}\in\mathcal{M}$, and $\alpha_{i}>0$ with $\sum\limits_{i=1}^{n}\alpha_{i}=1$ such that $\varphi(a)=\sum\limits_{i=1}^{n}\alpha_{i}T_{i}\pi_{i}(a)T_{i}^{-1}$, $\forall a\in \mathcal{A}$.
\end{enumerate}
\end{thm}

\begin{proof}
$(1)\Rightarrow(7)$, $(7)\Rightarrow(6)$, $(6)\Rightarrow(3)$, and $(1)\Rightarrow(5)$ are obvious.

$(1)\Rightarrow(2)$. If $\varphi$ is similar to a $*$-representation in $\mathcal{M}$, then there exists a positive invertible operator $T^{-1}\in\mathcal{M}$ such that $T^{-1}\varphi(\cdot)T=\pi(\cdot)$ is a $*$-representation. Consider
\begin{align*}
\Vert\Phi(\sum\limits_{i=1}^{n}a_{i}\otimes x_{i})\Vert&=\Vert\sum\limits_{i=1}^{n}\varphi(a_{i})x_{i}\Vert=\Vert\sum\limits_{i=1}^{n}T\pi(a_{i})T^{-1}x_{i}TT^{-1}\Vert\\
&\leq\Vert T\Vert\Vert T^{-1}\Vert\Vert\sum\limits_{i=1}^{n}\pi(a_{i})T^{-1}x_{i}T\Vert\\
&=\Vert T\Vert\Vert T^{-1}\Vert\Vert\sum\limits_{i=1}^{n}\pi(a_{i})x_{i}T^{-1}T\Vert\\
&=\Vert T\Vert\Vert T^{-1}\Vert\Vert\sum\limits_{i=1}^{n}\pi(a_{i})x_{i}\Vert\\
&\leq\Vert T\Vert\Vert T^{-1}\Vert\Vert\sum\limits_{i=1}^{n}a_{i}\otimes x_{i}\Vert_{\mathcal{A}\otimes_{max}\mathcal{M'}}.
\end{align*}
Thus $\Phi$ is bounded on a dense subset of $\mathcal{A}\otimes_{max}\mathcal{M'}$, hence $\Phi$ is bounded on $\mathcal{A}\otimes_{max}\mathcal{M'}$.

Consider the representation $$\Phi^{(n)}:\mathcal{A}\otimes \mathcal{M'}\otimes M_{n}(\mathbb{C})\longrightarrow \mathcal{B}(\mathcal{H})\otimes M_{n}(\mathbb{C})$$ defined by
$$\Phi^{(n)}(\sum\limits_{i=1}^k(a_{rs}^{i})\otimes x_{i})=\sum\limits_{i=1}^k(\varphi(a_{rs}^{i}))(x_{i}\otimes I_{n})\quad\forall~ a_{rs}^{i}\in\mathcal{A},x_{i}\in\mathcal{M'}.$$
Then
\begin{eqnarray*}
\Vert\Phi^{(n)}(\sum\limits_{i=1}^{k}(a_{rs}^{i})\otimes x_{i})\Vert
&=&\Vert\sum\limits_{i=1}^{k}(\varphi(a_{rs}^{i}))(x_{i}\otimes I_{n})\Vert\\
&=&\Vert\sum\limits_{i=1}^{n}(T\otimes I_{n})(\pi(a_{rs}^{i}))(T^{-1}\otimes I_{n})(x_{i}\otimes I_{n})(TT^{-1}\otimes I_{n})\Vert\\
&\leq&\Vert T\Vert\Vert T^{-1}\Vert\Vert\sum\limits_{i=1}^{n}(\pi(a_{rs}^{i}))(T^{-1}\otimes I_{n})(x_{i}\otimes I_{n})(T\otimes I_{n})\Vert\\
&=&\Vert T\Vert\Vert T^{-1}\Vert\Vert\sum\limits_{i=1}^{n}(\pi(a_{rs}^{i}))(x_{i}\otimes I_{n})(T^{-1}\otimes I_{n})(T\otimes I_{n})\Vert\\
&=&\Vert T\Vert\Vert T^{-1}\Vert\Vert\sum\limits_{i=1}^{n}(\pi(a_{rs}^{i}))(x_{i}\otimes I_{n})\Vert\\
&\leq&\Vert T\Vert\Vert T^{-1}\Vert\Vert\sum\limits_{i=1}^{n}(a_{rs}^{i})\otimes x_{i}\Vert_{\mathcal{A}\otimes_{max}\mathcal{M'}\otimes M_{n}(\mathbb{C})}.
\end{eqnarray*}
Thus $\Phi$ is completely bounded on a dense subset of $\mathcal{A}\otimes_{max}\mathcal{M'}$, hence $\Phi$ is a completely bounded representation on $\mathcal{A}\otimes_{max}\mathcal{M'}$.

$(2)\Rightarrow(1)$. If $\Phi$ is a completely bounded representation, then by Theorem~\ref{Haagerup}, $\Phi$ is similar to a $*$-representation in $\mathcal{B}(\mathcal{H})$. By Lemma~\ref{three equivalences}, there exists a positive invertible operator $T\in\mathcal{B}(\mathcal{H})$ such that for all $x\in\mathcal{A}\otimes_{max}\mathcal{M'}$, $T\Phi(x)=\Phi(x^{*})^{*}T$. In particular, for any $x\in\mathcal{M'}$, $Tx=xT$, so $T\in\mathcal{M''}$. By the Double Commutant Theorem for von Neumann algebras, $T\in\mathcal{M}$.

For any $a\in\mathcal{A}, 1\in\mathcal{M'}$, $$T\varphi(a)=T\Phi(a\otimes 1)=\Phi(a^{*}\otimes 1)^{*}T=\varphi(a^{*})^{*}T.$$
 By Lemma~\ref{three equivalences},  $\varphi$ is similar to a $*$-representation in $\mathcal{M}$.

$(1)\Rightarrow(3)$. If $\varphi$ is similar to a $*$-representation in $\mathcal{M}$, then there exists a positive invertible operator $T\in\mathcal{M}$ such that $T\varphi(\cdot)T^{-1}=\pi(\cdot)\in\mathcal{M}$ is a $*$-representation, so $\varphi(\cdot)=T^{-1}\pi(\cdot)T$.

For any $a\in\mathcal{A}$, consider \begin{align*}(T+T^{-1})\pi(a)(T+T^{-1})&=T\pi(a)T+T^{-1}\pi(a)T^{-1}+T^{-1}\pi(a)T+T\pi(a)T^{-1},\\
(T+iT^{-1})\pi(a)(T-iT^{-1})&=T\pi(a)T+T^{-1}\pi(a)T^{-1}+i(T^{-1}\pi(a)T-T\pi(a)T^{-1}).\end{align*}
Then
\begin{eqnarray*}
\varphi(a)&=&T^{-1}\pi(a)T\\&=&\frac{1}{2}[(T+T^{-1})\pi(a)(T+T^{-1})-(T\pi(a)T+T^{-1}\pi(a)T^{-1})]\\
& &+\frac{1}{2}i[(T\pi(a)T+T^{-1}\pi(a)T^{-1})-(T+iT^{-1})\pi(a)(T-iT^{-1})].
\end{eqnarray*}
Let \begin{align*}\varphi_{1}(a)&=\frac{1}{2}(T+T^{-1})\pi(a)(T+T^{-1}), \quad\varphi_{2}(a)=\frac{1}{2}(T\pi(a)T+T^{-1}\pi(a)T^{-1}), \\ \varphi_{3}(a)&=\frac{1}{2}(T\pi(a)T+T^{-1}\pi(a)T^{-1}), \quad\varphi_{4}(a)=\frac{1}{2}(T+iT^{-1})\pi(a)(T-iT^{-1}).\end{align*}
Then $\varphi_{1}, \varphi_{2}, \varphi_{3}, \varphi_{4}\in CP(\mathcal{A},\mathcal{M})$ and $\varphi=\varphi_{1}-\varphi_{2}+i(\varphi_{3}-\varphi_{4})$.

$(3)\Rightarrow(2)$. If $\varphi=\varphi_{1}-\varphi_{2}+i(\varphi_{3}-\varphi_{4})$, where $\varphi_{1}, \varphi_{2}, \varphi_{3}, \varphi_{4}\in CP(\mathcal{A},\mathcal{M})$.

Consider
\begin{align*}
\Vert\Phi(\sum\limits_{i=1}^{n}a_{i}\otimes x_{i})\Vert
&=\Vert\sum\limits_{i=1}^{n}\varphi(a_{i})x_{i}\Vert
=\Vert\sum\limits_{i=1}^{n}\varphi_{1}(a_{i})x_{i}+\varphi_{2}(a_{i})x_{i}+\varphi_{3}(a_{i})x_{i}+\varphi_{4}(a_{i})x_{i}\Vert\\
&\leq\Vert\sum\limits_{i=1}^{n}\varphi_{1}(a_{i})x_{i}\Vert+\Vert\sum\limits_{i=1}^{n}\varphi_{2}(a_{i})x_{i}\Vert+\Vert\sum\limits_{i=1}^{n}\varphi_{3}(a_{i})x_{i}\Vert+\Vert\sum\limits_{i=1}^{n}\varphi_{4}(a_{i})x_{i}\Vert\\
&\leq(\Vert\varphi_{1}\Vert+\Vert\varphi_{2}\Vert+\Vert\varphi_{3}\Vert+\Vert\varphi_{4}\Vert)\Vert\sum\limits_{i=1}^{n}a_{i}\otimes x_{i}\Vert_{\mathcal{A}\otimes_{max}\mathcal{M'}}.
\end{align*}
Thus $\Phi$ is bounded on a dense subset of $\mathcal{A}\otimes_{max}\mathcal{M'}$, hence $\Phi$ is bounded on $\mathcal{A}\otimes_{max}\mathcal{M'}$.

Consider the representation $$\Phi^{(n)}:\mathcal{A}\otimes_{max}\mathcal{M'}\otimes M_{n}(\mathbb{C})\longrightarrow \mathcal{B}(\mathcal{H})\otimes M_{n}(\mathbb{C})$$ defined by
$$\Phi^{(n)}(\sum\limits_{i}(a_{rs}^{i})\otimes x_{i})=\sum\limits_{i}(\varphi(a_{rs}^{i}))(x_{i}\otimes I_{n})\quad\forall~ a_{rs}^{i}\in\mathcal{A},x_{i}\in\mathcal{M'}.$$
Then
\begin{eqnarray*}
\Vert\Phi^{(n)}(\sum\limits_{i=1}^{k}(a_{rs}^{i})\otimes x_{i})\Vert
&=&\Vert\sum\limits_{i=1}^{k}(\varphi(a_{rs}^{i}))(x_{i}\otimes I_{n})\Vert\\
&=&\Vert\sum\limits_{i=1}^{n}(\varphi_{1}(a_{rs}^{i}))(x_{i}\otimes I_{n})+(\varphi_{2}(a_{rs}^{i}))(x_{i}\otimes I_{n})\\
& &+(\varphi_{3}(a_{rs}^{i}))(x_{i}\otimes I_{n})+(\varphi_{4}(a_{rs}^{i}))(x_{i}\otimes I_{n})\Vert\\
&\leq&\Vert\sum\limits_{i=1}^{n}(\varphi_{1}(a_{rs}^{i}))(x_{i}\otimes I_{n})\Vert +\Vert (\varphi_{2}(a_{rs}^{i}))(x_{i}\otimes I_{n})\Vert \\
& &+\Vert (\varphi_{3}(a_{rs}^{i}))(x_{i}\otimes I_{n})\Vert +\Vert (\varphi_{4}(a_{rs}^{i}))(x_{i}\otimes I_{n})\Vert\\
&\leq&(\Vert\varphi_{1}\Vert+\Vert\varphi_{2}\Vert+\Vert\varphi_{3}\Vert+\Vert\varphi_{4}\Vert)\Vert\sum\limits_{i=1}^{n}(a_{rs}^{i})\otimes x_{i}
\Vert_{\mathcal{A}\otimes_{max}\mathcal{M'}\otimes M_{n}(\mathbb{C})}.
\end{eqnarray*}
Thus $\Phi$ is completely bounded on a dense subset of $\mathcal{A}\otimes_{max}\mathcal{M'}$, hence $\Phi$ is a completely bounded representation on $\mathcal{A}\otimes_{max}\mathcal{M'}$.

$(5)\Rightarrow(3)$. If $\varphi: \mathcal{A}\rightarrow\mathcal{M}$, $\varphi(a)=T^{*}\phi(a)S$ is a completely bounded map, where $T, S\in\mathcal{M}$, $\phi: \mathcal{A}\rightarrow\mathcal{M}$ is a completely positive map, let \begin{align*}\varphi_{1}&=\frac{1}{4}(T^{*}+S^{*})\phi(T+S), \quad\varphi_{2}=\frac{1}{4}(T^{*}-S^{*})\phi(T-S),\\
\varphi_{3}&=\frac{1}{4}(T^{*}+iS^{*})\phi(T-iS), \quad\varphi_{4}=\frac{1}{4}(T^{*}-iS^{*})\phi(T+iS),\end{align*} then $\varphi_{1}, \varphi_{2}, \varphi_{3}, \varphi_{4}\in CP(\mathcal{A},\mathcal{M})$, and
$$\varphi_{1}-\varphi_{2}=\frac{1}{2}(T^{*}\phi S+S^{*}\phi T), \quad\varphi_{3}-\varphi_{4}=\frac{1}{2}i(-T^{*}\phi S+S^{*}\phi T).$$
Hence $$\varphi=T^{*}\phi S=\varphi_{1}-\varphi_{2}+i(\varphi_{3}-\varphi_{4}).$$

$(3)\Rightarrow(4)$. Suppose $\varphi=\sum\limits_{i=1}^{n}c_{i}\varphi_{i}$, $c_{i}\in\mathbb{C}$ and $\varphi_{i}\in CP(\mathcal{A},\mathcal{M})$, then $$\phi_{1}=\phi_{2}=\sum\limits_{i=1}^{n}\Vert c_{i}\Vert\varphi_{i}\in CP(\mathcal{A},\mathcal{M})$$ such that
$$\Psi: \mathcal{A}\rightarrow\mathcal{M}\otimes\mathcal{M}_{2}(\mathbb{C})$$ defined by
\begin{equation*}
\Psi(a)=\left ( \begin{matrix}\phi_{1}(a) &\varphi(a^{*})^{*}\\
\varphi(a)&\phi_{2}(a)
\end{matrix}\right )\quad\forall~a\in\mathcal{A}
\end{equation*}
is a completely positive map.

$(4)\Rightarrow(3)$. Let
\begin{align*}\varphi_{1}&=\frac{1}{4}(\phi_{1}+\phi_{2}+\varphi+\varphi^{*}),\quad \varphi_{2}=\frac{1}{4}(\phi_{1}+\phi_{2}-\varphi-\varphi^{*}),\\
\varphi_{3}&=\frac{1}{4}(\phi_{1}+\phi_{2}-i\varphi+i\varphi^{*}),\quad \varphi_{4}=\frac{1}{4}(\phi_{1}+\phi_{2}+i\varphi-i\varphi^{*}).\end{align*}
Then $\varphi_{1}, \varphi_{2}, \varphi_{3}, \varphi_{4}\in CP(\mathcal{A},\mathcal{M})$ and $\varphi=\varphi_{1}-\varphi_{2}+i(\varphi_{3}-\varphi_{4})$.
\end{proof}

\begin{thm}
Let \(\mathcal{A}\) be a separable \(C^{*}\)-algebra, and \(\mathcal{M}\) a von Neumann algebra with QWEP. Let \(\phi: \mathcal{A} \to \mathcal{M}\) be a completely bounded representation. Then there exists a positive invertible operator \(S \in \mathcal{M}\) such that \(S\phi(\cdot)S^{-1}\) is a \(*\)-representation.
\end{thm}

\begin{proof}
We may assume that $\mathcal{M}$ acts on \(\mathcal{H}\) faithfully, i.e., $\mathcal{M}\subset \mathcal{B}({\mathcal{H}})$. Since \(\mathcal{A}\) is separable, it is a quotient of \(C^{*}(F_{\infty})\). Let
\[
\phi_1: C^{*}(F_{\infty}) \longrightarrow \mathcal{A}
\]
be the natural surjective \(*\)-homomorphism. Then
\[
\phi \circ \phi_1: C^{*}(F_{\infty}) \longrightarrow \mathcal{M}
\]
is a completely bounded representation. Denote \(\phi_2 = \phi \circ \phi_1\).

Since \(C^{*}(F_{\infty})\) has the LP and \(\mathcal{M}\) is a von Neumann algebra with QWEP, and since the identity map \(\mathrm{id}: \mathcal{M}' \to \mathcal{M}'\) is a unital completely positive map, Theorem \ref{Pop} implies that \(\phi_2 \otimes id\) extends to a completely bounded representation
\[
C^{*}(F_{\infty}) \otimes_{\max} \mathcal{M}' \overset{\phi_2 \otimes id}{\longrightarrow} \mathcal{M} \otimes_{\max} \mathcal{M}',
\]
satisfying \(\|\phi_2\otimes id\|_{\mathrm{cb}} \leq \|\phi_2\|_{\mathrm{cb}}\).

Let $\pi:\,\mathcal{M}\otimes_{\max}\mathcal{M}'\longrightarrow \mathcal{B}(\mathcal{H})$ be the bounded unital $\ast$-representation such that
\[
\pi(x\otimes b)=xb,\quad x\in\mathcal{M}, b\in\mathcal{M}'.
\]
Define $\Phi:\, C^*(F_{\infty})\otimes_{\max}\mathcal{M}'\longrightarrow \mathcal{B}(\mathcal{H})$ by
\[
\Phi(\sum_{i=1}^n x_i\otimes b_i)=\sum_{i=1}^n\phi_2(x_i)b_i, \quad x_i\in C^*(F_{\infty}), b_i\in \mathcal{M}'.
\]
Then $\Phi=\pi\circ (\phi_2\otimes id)$ is a completely bounded representation of $ C^*(F_{\infty})\otimes_{\max}\mathcal{M}'$ on $\mathcal{H}$.
 By Theorem \ref{Haagerup}, \(\Phi\) is similar to a \(*\)-representation in $\mathcal{B}(\mathcal{H})$. By  Lemma \ref{three equivalences},  there exists a positive invertible operator \(T \in \mathcal{B}(\mathcal{H})\) such that
\[
T \Phi\Bigl(\sum_{i=1}^n x_i \otimes b_i\Bigr) = \Phi\Bigl(\sum_{i=1}^n x_i^* \otimes b_i^*\Bigr)^* T, \quad \forall x_i \in C^{*}(F_{\infty}), \; b_i \in \mathcal{M}'.
\]
In particular, for any \(b \in \mathcal{M}'\),
\[
T \Phi(1 \otimes b) = \Phi(1 \otimes b^*)^* T.
\]
Since \(\Phi(1 \otimes b) = b\) and \(\Phi(1 \otimes b^*) = b^*\), this implies \(T b = b T\). Hence, \(T\) commutes with every element of \(\mathcal{M}'\), so \(T \in \mathcal{M}''\). By the Double Commutant Theorem, \(\mathcal{M}'' = \mathcal{M}\), thus \(T \in \mathcal{M}\). Moreover, for any \(x \in C^{*}(F_{\infty})\),
\[
T \phi_2(x) = \phi_2(x^*)^* T.
\]

Since \(\phi_1\) is a surjective \(*\)-homomorphism, for any \(a \in \mathcal{A}\) there exists an \(x \in C^{*}(F_{\infty})\) such that \(\phi_1(x) = a\) and \(\phi_1(x^*) = a^*\). Therefore,
\[
T \phi(a) = T \phi(\phi_1(x)) = T \phi_2(x) = \phi_2(x^*)^* T = (\phi \circ \phi_1(x^*))^* T = \phi(a^*)^* T, \quad \forall a \in \mathcal{A}.
\]

By Lemma \ref{three equivalences}, setting \(S = T^{1/2} \in \mathcal{M}\) yields that \(S \phi(\cdot) S^{-1}\) is a \(*\)-representation. This completes the proof.
\end{proof}

\section{A second approach to the main theorem}
The following Lemma is due to Don Hadwin (see the proof of (1)$\Rightarrow$(3) of Theorem 3)~\cite{Hadwin}.
\begin{lem}{\rm (Hadwin)}\label{1}
If $\phi:\mathcal{A}\rightarrow\mathcal{B}(\mathcal{H})$, then the following are equivalent:
\begin{enumerate}
\item [{\rm (1)}]$\phi$ is a unital completely bounded linear map from $\mathcal{A}$ to $\mathcal{B}(\mathcal{H})$;
\item [{\rm (2)}]There exists a unital $*$-representation $\pi:\mathcal{A}\rightarrow\mathcal{B}(\mathcal{H})$ and operators $X,Y\in\mathcal{B}(\mathcal{H})$, such that $XY=I$ and $\phi(a)=X\pi(a)Y$ for every a in $\mathcal{A}$.
    \end{enumerate}
\end{lem}

\begin{lem}\label{2}
If $\phi:\mathcal{A}\rightarrow\mathcal{B}(\mathcal{H})$ is a unital completely bounded linear map, then there exists an invertible positive operator $T\in\mathcal{B}(\mathcal{H})$, such that $T\phi(a)=\phi(a^{*})^{*}T$.
\end{lem}

\begin{proof}
By lemma \ref{1}, there exists a unital $*$-representation $\pi:\mathcal{A}\rightarrow\mathcal{B}(\mathcal{H})$ and operators $X,Y\in\mathcal{B}(\mathcal{H})$, such that $XY=I$ and $\phi(a)=X\pi(a)Y$ for every a in $\mathcal{A}$. Since $XY=I$, for every $\xi\in\mathcal{H}$, $$\Vert\xi\Vert=\Vert XY\xi\Vert\leq\Vert X\Vert\Vert Y\xi\Vert,$$ that is $\Vert Y\xi\Vert\geq\frac{1}{\Vert X\Vert}\Vert\xi\Vert$. Therefore $$\Vert|Y|\xi\Vert=\Vert Y\xi\Vert\geq\frac{1}{\Vert X\Vert}\Vert\xi\Vert.$$ This implies that $|Y|$ is injective and with closed range. Thus $|Y|$ is invertible. Since $XY=I$, $Y^{*}X^{*}=I$. Similarly, $|X^{*}|$ is invertible.

Let $Y=V_{1}|Y|$, $X=|X^{*}|V_{2}$ be polar decompositions. Then $XY=I$ implies $|X^{*}|V_{2}V_{1}|Y|=I$. Since $|Y|$, $|X^{*}|$ are invertible, $V_{2}V_{1}$ is invertible. Therefore, $RanV_{2}=\mathcal{H}$ and $V_{2}V_{2}^{*}=I$. Note that $V_{1}^{*}V_{2}^{*}$ is also invertible. So $V_{1}^{*}V_{1}=I$. Let $T=|X^{*}|^{-2}$. Claim $T\phi(a)=\phi(a^{*})^{*}T$. Since $|X^*| V_2 V_1 |Y| = I$, we have $ I = |Y| V_1^* V_2^* |X^*|$ and therefore  $|X^*|^{-1} V_2 = |Y|V_1^*.$
Since \begin{align*}\phi(a)&=X\pi(a)Y=|X^{*}|V_{2}\pi(a)V_{1}|Y|,\\
\phi(a^{*})^{*}&=|Y|V_{1}^{*}\pi(a)V_{2}^{*}|X^{*}|,\end{align*}
we have  $$|X^*|^{-1} V_2 \pi(a) V_1|Y| = |Y|V_1^* \pi(a) V_2^* |X^*|^{-1}$$ and therefore
$$
T\phi(a) = \phi(a^*)^* T.
$$
This completes the proof.
\end{proof}

For a proof of the following Lemma, we refer to Proposition 1.48 in \cite{pisier}.
\begin{lem}\label{3}
Let \((\mathcal{A}_i)_{i \in I}\) be a family of \(C^*\)-algebras.
Then \((\oplus\sum_{i \in I} \mathcal{A}_i)_{\infty}\) is injective iff \(\mathcal{A}_i\) is injective for all \(i \in I\).
In particular, \((\oplus\sum_{n=1}^{\infty} M_{n}(\mathbb{C}))_{\infty}\) is a unital injective \(C^*\)-algebra.
\end{lem}

For a proof of the following Lemma, we refer to Theorem 4.11 in \cite{pisier}
\begin{lem}\label{4}
There is a surjective unital completely positive map
\[ P_G : L(G) \otimes_{\max} L(G) \rightarrow C^*(G). \]
 Moreover, for all \( a, b \in L(G) \), where \( a(\delta_e) = \sum_{t\in G} a(t) \delta_t\) and \( b(\delta_e) = \sum_{t\in G} b(t) \delta_t\), we have
\[ P_G(a \otimes b) = \sum_{t\in G} a(t)b(t) U_G(t). \]

In particular, the restriction \( P_G|_{L(G) \otimes \mathbb{C}} \rightarrow C^*(G) \) is an into unital completely positive map and
\( \text{Ran}(P_G|_{L(G) \otimes \mathbb{C}}) \) contains \( \{U_G(t) \mid t \in G\} \).
\end{lem}

\begin{lem}\label{5}
There exists a  {\rm (nonseparable)} Hilbert space \( \mathcal{K} \) and a unital completely positive map \( \theta : \mathcal{B}(\mathcal{K}) \rightarrow C^*(F_\infty) \) such that
\( \text{Ran}(\theta) \text{ contains } \{U_{F_\infty}(t) \mid t \in F_\infty\}. \)
\end{lem}

\begin{proof}
By Lemma \ref{3}, there exists a (nonseparable) Hilbert space \( \mathcal{K} \) and a conditional expectation \( E_1 \) from \( \mathcal{B}(\mathcal{K}) \) onto
\( (\oplus\sum_{n=1}^{\infty} M_n(\mathbb{C}))_{\infty}. \)
By definition, there exists an onto $^*$-homomorphism
\[ \pi : (\oplus\sum_{n=1}^{\infty} M_n(\mathbb{C}))_{\infty} \rightarrow \prod_{n=1}^{\omega}M_n(\mathbb{C}). \]

Since \( L(F_\infty) \) can be embedded into \( \prod_{n=1}^{\omega}M_n(\mathbb{C}) \), there is a conditional expectation \( E_2 \) from \( \prod_{n=1}^{\omega}M_n(\mathbb{C}) \) onto \( L(F_\infty) \).
By Lemma \ref{4}, there exists a unital completely positive map \( P_{F_\infty} : L(F_\infty) \rightarrow C^*(F_\infty) \) such that \( \text{Ran}(P_{F_\infty}) \) contains \( \{U_{F_\infty}(t) \mid t \in F_\infty\} \). Let $\theta=P_{F_\infty}\circ E_2\circ\pi\circ E_1$. Then $\theta$ satisfies Lemma \ref{5}.
\end{proof}

\begin{thm}\label{7}
Let \(\phi : \mathcal{A} \rightarrow \mathcal{M}\) be a unital completely bounded representation from a separable \(C^*\)-algebra \(\mathcal{A}\) into a separable von Neumann algebra \(\mathcal{M}\) with QWEP. Then there exists a positive invertible operator \(S \in \mathcal{M}\) such that \(S\phi(\cdot)S^{-1}\) is a \(*\)-representation.
\end{thm}
\begin{proof}
We may assume \(\mathcal{A}=C^*(F_\infty)\) and \(\mathcal{M} \) act on \(\mathcal{H}\) in standard form. There exists a weakly dense \(C^*\)-subalgebra \( \mathcal{B}\subseteq \mathcal{M'}\) and a surjective $\ast$-homomorphism $\varphi: C^*(F_\infty)\rightarrow\mathcal{B}$. By Lemma \ref{5}, there exists a Hilbert space \(\mathcal{K}\) and a unital completely positive map \(\theta : \mathcal{B}(\mathcal{K}) \rightarrow C^*(F_\infty)\) such that \( \text{Ran}(\theta) \) contains \( \{U_{F_\infty}(t) \mid t \in F_\infty\} \).

Now consider the following diagram:
\[
C^*(F_\infty) \otimes_{max} \mathcal{B}(\mathcal{K}) \overset{\phi \otimes id}{\longrightarrow} \mathcal{M} \otimes_{max} \mathcal{B}(\mathcal{K}) \overset{id \otimes \theta}{\longrightarrow} \mathcal{M} \otimes_{max} C^*(F_\infty) \overset{\pi}{\longrightarrow} \mathcal{B}(\mathcal{H}).
\]
where \(\pi(x \otimes b) = x \varphi(b)\) is a unital bounded \(*\)-representation. Note that \(id \otimes \theta\) is a unital completely positive map. By Theorem~\ref{Pop}, \(\phi \otimes id\) extends to a completely bounded representation.

Define $$\Phi: C^*(F_\infty)\otimes_{max}\mathcal{B}(\mathcal{K})\rightarrow\mathcal{B}(\mathcal{H})$$ by $$\Phi(a\otimes x)=\phi(a)\varphi(\theta(x)).$$ Then $\Phi=\pi\circ (id\otimes \theta)\circ (\phi\otimes id)$ is a unital completely bounded linear map. By Lemma \ref{2}, there exists an invertible positive operator $T\in\mathcal{B}(\mathcal{H})$ such that $T\Phi(a\otimes x)=\Phi(a^*\otimes x^*)^*T$. Let $a=1$. Then $$T\varphi(\theta(x))=\varphi(\theta(x))T$$ for all $x\in\mathcal{B}(\mathcal{K})$. Note that $Ran(\theta)$ contains \( \{U_{F_\infty}(t) \mid t \in F_\infty\} \). Therefore, $Ran\varphi(\theta(x))$ generate $\mathcal{M}'$. Hence $T\in(\mathcal{M}')'=\mathcal{M}$. Let $x=1$. Then $T \phi(a) = \phi(a^*)^* T$ for all \(a \in A\). Put $S=T^{1/2}\in \mathcal{M}$. By Lemma~\ref{three equivalences}, $T\phi(\cdot)T^{-1}$ is a $\ast$-representation.
\end{proof}

\section{The classical similarity problem revisited}
The following lemma is due to Uffe Haagerup~\cite{Haagerup}.
\begin{lem}{\rm (Haagerup)}\label{Haagerup11}
Let \(\varphi\) be a cyclic bounded representation of a \(C^*\)-algebra \(\mathcal{A}\) on a Hilbert space \(\mathcal{H}\). Then there exists an invertible positive operator \(T \in \mathcal{B}(\mathcal{H})\) and a \(*\)-representation \(\pi\) of \(\mathcal{A}\) into \(\mathcal{B}(\mathcal{H})\) such that
\[
\|T\| \cdot \|T^{-1}\| \leq \|\varphi\|^3 \quad \text{and} \quad \varphi(a) = T^{-1} \pi(a) T \quad \forall a \in A.
\]
\end{lem}
\begin{thm}
Let \(\phi: C^*(F_\infty) \to \mathcal{M} = \mathcal{B}(l^{2})\) be a bounded representation. Assume that $\mathcal{M}$ acts on \(\mathcal{H}\) and \(\xi\in H\) be a cyclic and separate vector of \(\mathcal{M}\). Let \(\mathcal{B} \subseteq \mathcal{M}'\) be a weakly dense \(C^*\)-algebra. Then the following conditions are equivalent:
\begin{enumerate}
    \item[{\rm (1)}] \(\phi\) is similar to a \(*\)-representation from \(C^*(F_\infty)\) into \(\mathcal{B}(l^{2})\);
    \item[{\rm (2)}] \(\Phi : C^*(F_\infty) \otimes_{\max} \mathcal{B} \to \mathcal{B}(\mathcal{H})\) given by
    \[
    \Phi\left( \sum_{i=1}^n a_i \otimes b_i \right) = \sum_{i=1}^n \phi(a_i) b_i
    \]
    is a bounded representation.
\end{enumerate}
\end{thm}

\begin{proof}
\((1) \Rightarrow (2)\). Note that \(\phi(a) = T^{-1} \pi(a) T \in \mathcal{M} = \mathcal{B}(l^{2})\), where the invertible positive operator \(T\) and \(*\)-representation \(\pi(a) \in M\) for all \(a \in C^*(F_\infty)\), then
\begin{eqnarray*}
\Vert\Phi\left( \sum_{i=1}^n a_i \otimes b_i \right)\Vert
&=& \Vert\sum_{i=1}^n \phi(a_i) b_i\Vert\\
&=& \Vert\sum_{i=1}^n T^{-1} \pi(a_i) T b_i\Vert\\
&=&\Vert\sum_{i=1}^n T^{-1}\pi(a_i) b_i T\Vert\\
&\leq&\Vert T\Vert\Vert T^{-1}\Vert\Vert\sum_{i=1}^n a_i \otimes b_i\Vert_{C^*(F_\infty) \otimes_{\max} \mathcal{B}}.
\end{eqnarray*}
This implies that \(\Phi\) is a bounded representation.

\((2) \Rightarrow (1)\). Since $$[\Phi(C^*(F_{\infty}) \otimes_{\max} \mathcal{B})\xi] \supseteq [\mathcal{B}\xi]=[\mathcal{M}'\xi]=\mathcal{H},$$
$\Phi$ is a bounded cyclic representation. By Lemma \ref{Haagerup11}, there exists an invertible positive operator $T\in\mathcal{B}(\mathcal{H})$ such that $$T\Phi(a\otimes b)=\Phi(a^*\otimes b^*)^*T.$$ Let $a=1$. Then $Tb=bT$ for all $b\in\mathcal{B}$. Since $\mathcal{B}$ is wealy dense in $\mathcal{M}'$, $T\in(\mathcal{M}')'=\mathcal{M}$. Let $b=1$. Then $$T\phi(a)=\phi(a^*)^*T$$ for all $a\in\mathcal{A}$. By Lemma~\ref{three equivalences}, \(\phi\) is similar to a \(*\)-representation from \(C^*(F_\infty)\) into \(\mathcal{B}(l^{2})\).
\end{proof}

\begin{prop}
Let $\rho_{1}$ be a state on \(\mathcal{B}(l^{2})\), $\rho_{2}$ be a state on a nuclear \(C^*\)-algebra \(\mathcal{B}\). Suppose $\phi: C^*(F_\infty) \to \mathcal{B}(l^{2})$ is a bounded representation and \(\Phi : C^*(F_\infty) \otimes \mathcal{B} \to \mathcal{B}(l^{2})\otimes\mathcal{B}\) is given by
    \[
    \Phi\left( \sum_{i=1}^n a_i \otimes b_i \right) = \sum_{i=1}^n \phi(a_i) \otimes b_i.\]
Then $\rho_{1}\otimes\rho_{2}((\sum_{i=1}^n \phi(a_i) \otimes b_i)^*(\sum_{i=1}^n \phi(a_i) \otimes b_i))^{1/2}\leq\Vert\phi\Vert^{3}\Vert\sum_{i=1}^n a_i \otimes b_i\Vert$.
\end{prop}

\begin{proof}
Consider the GNS construction, $\pi_{\rho_{1}}: \mathcal{B}(l^{2})\to \mathcal{B}(\mathcal{H}_{1})$ and $\pi_{\rho_{2}}: \mathcal{B}\to \mathcal{B}(\mathcal{H}_{2})$ are \(*\)-representation.  Then $\pi_{\rho_{1}}\otimes\pi_{\rho_{2}}: \mathcal{B}(l^{2})\otimes\mathcal{B}\to\mathcal{B}(\mathcal{H}_{1})\overline{\otimes}\mathcal{B}(\mathcal{H}_{2})$ is a \(*\)-representation.  There exists unit vectors $\xi\in\mathcal{H}_{1}$ and $\eta\in\mathcal{H}_{2}$ such that

\[
\begin{aligned}
& \rho_{1}\otimes\rho_{2}((\sum_{i=1}^n \phi(a_i) \otimes b_i)^*(\sum_{i=1}^n \phi(a_i) \otimes b_i))\\
&= \langle(\sum_{i=1}^n \varphi(a_i) \otimes b_i)^*(\sum_{i=1}^n \phi(a_i) \otimes b_i)(\xi\otimes\eta), \xi\otimes\eta\rangle \\
&= \Vert\sum_{i=1}^n \phi(a_i)\xi \otimes b_i\eta\Vert^{2}.
\end{aligned}
\]

Let $\mathcal{K}_{1}=[\phi(C^*(F_{\infty}))\xi]\subseteq\mathcal{H}_{1}$. By Lemma \ref{Haagerup11}, $\phi(a) = T^{-1} \pi(a) T$, where $T, \pi(a)\in\mathcal{B}(\mathcal{K}_{1})$ and \(\|T\| \cdot \|T^{-1}\| \leq \|\phi\|^3\), so
\begin{eqnarray*}
\Vert\sum_{i=1}^n \phi(a_i)\xi \otimes b_i\eta\Vert^{2}
&\leq& \Vert\sum_{i=1}^n (T^{-1} \pi(a_i) T \otimes b_i)(\xi\otimes\eta)\Vert^{2}\\
&\leq& \Vert\sum_{i=1}^n T^{-1} \pi(a_i) T \otimes b_i\Vert^{2}_{\mathcal{B}(\mathcal{K}_{1})\overline{\otimes}\mathcal{B}(\mathcal{H}_{2})}\\
&\leq&\Vert T^{-1}\Vert^{-2}\Vert T\Vert^{2}\Vert\sum_{i=1}^n \pi(a_i) \otimes b_i\Vert^{2}\\
&\leq&\|\phi\|^6\Vert \sum_{i=1}^na_i \otimes b_i\Vert^{2}.
\end{eqnarray*}
This completes the proof.
\end{proof}

\section{A Concluding Remark}

Combining Theorem~\ref{Pis} and Theorem~\ref{Main} and note that $C^*{(F_\infty)}$ is not nuclear, we conclude that not every von Neumann algebra is QWEP. By Theorem 14.1 of~\cite{pisier}, this implies that not every separable tracial von Neumann algebra can be embedded into $\mathcal{R}^w$, the ultrapower algebra of the hyperfinite type ${\rm II}_1$ factor $\mathcal{R}$. This provides a negative answer to the Connes Embedding Problem.

\section{Appendix}
In the appendix, we give a proof Lemma 2.3 of~\cite{Pop} for reader's convenience. The idea is due to Professor Florin Pop. 

Let $\mathcal{B}$ be a $C^*$-algebra and $\mathcal{J}$ be a closed ideal of $\mathcal{B}$. We will consider $\mathcal{B}$ in its universal representation $\pi_u$ on the Hilbert space $\mathcal{H}$. Then we can identify $\pi_u(\mathcal{B})''$ with $\mathcal{B}^{**}$. There is a central projection $P$ in $\mathcal{B}^{**}$ such that $(I-P)\mathcal{B}^{**}=\mathcal{J}^{**}$, $P\mathcal{B}^{**}=(\mathcal{B}/\mathcal{J})^{**}$. We will therefore view $\mathcal{B}/\mathcal{J}$ universally represented on the Hilbert space $P\mathcal{H}$.

\begin{lem}\label{Appendix}
$\mathcal{B}\cap \mathcal{J}^{**}=\mathcal{J}$.
\end{lem}
\begin{proof}
Clearly we only need to prove that $\mathcal{B}\cap \mathcal{J}^{**}\subseteq \mathcal{J}$. Otherwise, there is an element $b\in \mathcal{B}\cap \mathcal{J}^{**}$ such that $b\notin\mathcal{J}$. By the Hahn-Banach separation theorem, there is a bounded linear functional $\varphi$ on $\mathcal{B}$ such that $\varphi(\mathcal{J})=0$ and $\varphi(b)\neq 0$. Then $\varphi$ induces a bounded linear functional $\phi$ on $\mathcal{B}/\mathcal{J}$ such that $\phi(\rho(b))\neq 0$, where $\rho: \,\mathcal{B}\rightarrow\mathcal{B}/\mathcal{J}$ is the quotient map. By Remark 10.1.3 of~\cite{Kadison}, there are vectors $\xi,\eta\in P\mathcal{H}$ such that
\[
\phi(\rho(b))=\langle b\xi,\eta\rangle=\langle bP\xi,P\eta\rangle=0
\]
since $b\in \mathcal{J}^{**}$ and $Pb=bP=0$. This contradicts to $\phi(\rho(b))\neq 0$.
\end{proof}

\begin{cor}\label{Appendix2}
$P\mathcal{B}=\mathcal{B}/\mathcal{J}$ such that $Pb=\rho(b)$ for $b\in \mathcal{B}$, where $\rho: \,\mathcal{B}\rightarrow\mathcal{B}/\mathcal{J}$ is the quotient map.
\end{cor}
\begin{proof}
Consider the surjective $\ast$-homomorphism $\phi:\, \mathcal{B}\rightarrow P\mathcal{B}$ defined by $\phi(b)=Pb$ for $b\in \mathcal{B}$. Then $\phi(b)=0$ if and only if $(I-P)b=b$, i.e., $b\in \mathcal{J}^{**}$. By Lemma~\ref{Appendix}, $b\in \mathcal{J}$. Thus $\phi=\rho$.
\end{proof}

The following theorem is Lemma 2.3 of~\cite{Pop}, which plays a crucial role in the proof of Theorem~\ref{Pop}.
\begin{thm}{\rm (Pop)}
Let $\mathcal{A}$ be a $C^*$-algebra with the LP and $\pi:\,\mathcal{A}\rightarrow\mathcal{B}/\mathcal{J}$ a completely bounded homomorphism. Then there exists a completely bounded map $\beta:\,\mathcal{A}\rightarrow \mathcal{B}$ such that $\|\beta\|_{cb}\leq \|\pi\|_{cb}$ and $\pi=\rho\circ\beta$, where $\rho:\, \mathcal{B}\rightarrow\mathcal{B}/\mathcal{J}$ is the quotient map.
\end{thm}
\begin{proof}
Since $\mathcal{B}^{**}=(\mathcal{B}/\mathcal{J})^{**}\oplus \mathcal{J}^{**}$, we will consider $\mathcal{B}$ in its universal representation on the Hilbert space $\mathcal{H}$ and denote by $P$ the central projection in $\mathcal{B}^{**}$ such that $(I-P)\mathcal{B}^{**}=\mathcal{J}^{**}$, $P\mathcal{B}^{**}=(\mathcal{B}/\mathcal{J})^{**}$.  We will therefore view $\mathcal{B}/\mathcal{J}$ universally represented on the Hilbert space $P\mathcal{H}$. An operator $X\in\mathcal{B}(P\mathcal{H})$ can be naturally viewed as an operator $X\oplus 0$ in $\mathcal{B}(\mathcal{H})$. By Corollary~\ref{Appendix2}, $P\mathcal{B}=\rho(\mathcal{B})$ such that $Pb=\rho(b)$ for $b\in \mathcal{B}$. Let $S\in \mathcal{B}(P\mathcal{H})$ be an invertible operator and $\theta:\, \mathcal{A}\rightarrow \mathcal{B}(P\mathcal{H})$ be a $\ast$-homomorphism such that $\pi(\cdot)=S^{-1}\theta(\cdot) S$ and $\|S^{-1}\|\|S\|=\|\pi\|_{cb}$. We can assume (by rescaling, if necessary) that $\|S\|=1$. Consider the invertible operator $T\in \mathcal{B}(\mathcal{H})$ defined as
$T=S\oplus (I-P)$. It is immediate that $\|T\|=1$, $\|T^{-1}\|=\|S^{-1}\|$, and $TjT^{-1}=j$ for all $j\in \mathcal{J}$. It follows that the non-selfadjoint algebra $T\mathcal{B}T^{-1}$ contains $\mathcal{J}$ as an ideal.

Consider the quotient map: $\tilde{\rho}(TbT^{-1})=TPbT^{-1}=T \rho(b) T^{-1}$ from $T\mathcal{B}T^{-1}$ onto $T\mathcal{B}T^{-1}/\mathcal{J}$. Note that $T\mathcal{B}T^{-1}/\mathcal{J}$ is a unital nonselfadjoint algebra with unit $P$.  Since \[\pi=S^{-1}\theta S=T^{-1}\theta T,\]
\begin{equation}\label{E:1}
\theta=T\pi T^{-1}
\end{equation} is a unital $\ast$-homomorphism from $\mathcal{A}$ into $T\mathcal{B}T^{-1}/\mathcal{J}$ and $\mathcal{J}$ is selfadjoint.  Let $\mathcal{C}=T\mathcal{B}T^{-1}$ and $\mathcal{C}^*=\{X^*:\, X\in \mathcal{C}\}$. Then $\mathcal{C}$ and $\mathcal{C}^*$ are norm closed unital algebras in $\mathcal{B}(\mathcal{H})$. Since $\mathcal{J}$ is selfadjoint, $\mathcal{J}\subset\mathcal{C}^*$. Note that $\theta(\mathcal{A})\subset \mathcal{C}^*/\mathcal{J}$. Note that $(\mathcal{C}\cap\mathcal{C}^*)/\mathcal{J}=(\mathcal{C}/\mathcal{J})\cap (\mathcal{C}^*/\mathcal{J})$ (see Lemma 4.2 of~\cite{Pop}). Let $\mathcal{D}=\mathcal{C}\cap\mathcal{C}^*$. Then $\mathcal{D}$ is a unital $C^*$-subalgebra of $\mathcal{B}(\mathcal{H})$ and $\mathcal{J}$ is an ideal of $\mathcal{D}$. Note that $\theta(\mathcal{A})\subset \mathcal{D}/\mathcal{J}$ and $\theta$ is a unital $\ast$-homomorphism from $\mathcal{A}$ into $\mathcal{D}/\mathcal{J}$.

Since $\mathcal{A}$ has the LP, there exists a unital completely positive map $\alpha:\, \mathcal{A}\rightarrow\mathcal{D}\subset T\mathcal{B}T^{-1}$ such that $\theta=\tilde{\rho}\circ \alpha$. Define $\beta=T^{-1}\alpha T:\, \mathcal{A}\rightarrow\mathcal{B}$. By equation~(\ref{E:1}),
\[
\rho\circ \beta=\rho\circ (T^{-1}\alpha T)=T^{-1}\tilde{\rho}(TT^{-1}\alpha TT^{-1})T=T^{-1}(\tilde{\rho}\circ \alpha)T=T^{-1}\theta T=\pi.
\]
 This completes the proof.
\end{proof}


\begin{thebibliography}{99}
\bibitem{Barnes} B. A. Barnes, {\it The similarity problem for representations of a $B^*$-algebras}, Mich. Math. J., {\bf 22} (1975), 25-32.
\bibitem{B-O} N. Brown and N. Ozawa, {\it $C^{*}$-algebras and finite-dimensional approximations}, Providence, RI: Amer. Math. Soc., 2008.
\bibitem{Bunce} J. W. Bunce, {\it The similarity problem for representations of $C^*$-algebras}, Proc. A. M. S., {\bf 81} (1981), 409-414.
\bibitem{C-E} M. D. Choi and E. G. Effros, {\it Nuclear $C^*$-algebras and injectivity. The general case}, Ind. Univ. Math. J., {\bf 26} (1977), 443-446.
\bibitem{Connes} A. Connes, {\it Classification of injective factors}, Ann. of Math., {\bf 104} (1976), 73-116.
\bibitem{Christensen} E. Christensen, {\it On non self-adjoint representations of $C^{*}$-algebras}, Amer. J. Math., {\bf103}(1981), 817-833.
\bibitem{Chr} E. Christensen, {\it Similarities of $II_{1}$ factors with property $\Gamma$}, J. Operator Theory, {\bf15}(1986), 281-288.
\bibitem{Haagerup} U. Haagerup, {\it Solution of the similarity problem for cyclic representations of $C^{*}$-algebras}, Ann.  Math., {\bf118}(1983), 215-240.
\bibitem{Hadwin} D. Hadwin, {\it Dilations and hahn decompositions for linear maps}, Can. J. Math., {\bf33}(1981), 826-839.
\bibitem{Ji} Z. Ji, A. Natarajan, T. Vidick, J. Wright, and H. Yuen, {\it MIP* = RE}, Preprint, arXiv:2001.04383, 2021.
\bibitem{Kad} R. V. Kadison, {\it On the orthogonlaizition of operator representations}, Amer J. Math., {\bf 77} (1955),600-620.
\bibitem{Kadison} R. V. Kadison and J. Ringrose, {\it Fundamentals of the theory of operator algebras, Vol. ${\rm II}$}, Birkh\(\ddot{a}\)user Boston, Inc., Boston, MA,
    1992.
\bibitem{Kir} E. Kirchberg, {\it On nonsemisplit extensions, tensor products and exactness of group $C^*$-algebras}.
Invent. Math., {\bf 112} (1993), 449-489.
\bibitem{Lan} C. Lance, {\it On nuclear $C^*$-algebras}, J. Funct. Anal., {\bf 12} (1973) 157-176.
\bibitem{Oz} N. Ozawa, {\it An invitation to the similarity problems}, Preprint.
\bibitem{Ozawa} N. Ozawa, {\it About the QWEP conjecture}, Internat. J. Math., {\bf15}(2004), 501-530.
\bibitem{Pis} G. Pisier, {\it Simultaneous similairty, bounded generation and amenability}, Tohoku Math. J., {\bf 59} (2007), 79-99.
\bibitem{pisier} G. Pisier, {\it Tensor products of $C^*$-algebras and operator spaces-the Connes-Kirchberg problem}, Cambridge: Cambridge University Press, 2020.
\bibitem{Pop} F. Pop, {\it Similarities of tensor products of type $ II_{1}$ factors}, Integral Equations Operator Theory, {\bf 89}(2017), 455-463.
\bibitem{V-Z} F. H. Vasilescu and L. Zsido, {\it Uniformly bounded groups in finite $W^*$-algebras}, Acta
Sci. Math. (Szeged), 36 (1974), 189-192.
\bibitem{Wit} G. Wittstock, {\it Ein operatorwertiger Hahn-Banach Satz}, J. Funct. Anal., {\bf 40}(1981), 127-150.

\end{thebibliography}
\end{document}